\documentclass[12pt]{article}
\usepackage{amsmath,amsthm,amssymb}
\usepackage{amssymb}
\newtheorem{theorem}{Theorem}[section]
\newtheorem{corollary}[theorem]{Corollary}

\usepackage[margin=35mm]{geometry}

\DeclareMathOperator{\e}{\mathrm{e}}

\begin{document}
\title{A Digit Reversal Property for an Analogue of Stern's Sequence}
\author{Lukas Spiegelhofer
\footnote{The author acknowledges support by the Austrian Science Fund (FWF), project F5502-N26, which is a part of the Special Research Program ``Quasi Monte Carlo Methods: Theory and Applications''.}}
\date{\small Institute of Discrete Mathematics and Geometry\\ Vienna University of Technology\\ Wiedner Hauptstra\ss e 8--10, Vienna, Austria\\
\texttt{lukas.spiegelhofer@tuwien.ac.at }}
\maketitle
\renewcommand{\thefootnote}{\fnsymbol{footnote}}
\renewcommand{\thefootnote}{\arabic{footnote}}


\maketitle
\begin{abstract}
We consider a variant of Stern's diatomic sequence, studied recently by Northshield.
We prove that this sequence $b$ is invariant under \emph{digit reversal} in base $3$,
that is, $b_n=b_{n^R}$, where $n^R$ is obtained by reversing the base-$3$ expansion of $n$.
\end{abstract}
\section{Introduction}
\emph{Stern's diatomic sequence} $s$ is defined by $s_0=0$, $s_1=1$
and
\[s_{2n}=s_n,\qquad s_{2n+1}=s_n+s_{n+1}\]
for all $n\geq 1$.
Northshield~\cite{N2015} introduced the following analogue having values in $\mathbb Z[\sqrt{2}]$:
$b_0=0$, $b_1=1$ and
\begin{align*}
b_{3n}&=b_n,\\
b_{3n+1}&=\tau\cdot b_n+b_{n+1},\\
b_{3n+2}&=b_n+\tau \cdot b_{n+1}
\end{align*}
for $n\geq 0$, where $\tau=\sqrt{2}$.
The first values of the sequence $b$ are therefore as follows.
\[\begin{array}{ccccccccccccccc}
n&0&1&2&3&4&5&6&7&8&9&10&11&12&13\\
b_n&0&1&\sqrt{2}&1&2\sqrt{2}&3&\sqrt{2}&3&2\sqrt{2}&1&3\sqrt{2}&5&2\sqrt{2}&7\\[2mm]
n&14&15&16&17&18&19&20&21&22&23&24&25&26&27\\
b_n&5\sqrt{2}&3&4\sqrt{2}&5&\sqrt{2}&5&4\sqrt{2}&3&5\sqrt{2}&7&2\sqrt{2}&5&3\sqrt{2}&1
\end{array}
\]
(for a longer list consult Northshield's paper~\cite{N2015}, for example).
Northshield proved that
\[
\limsup_{n\rightarrow\infty}\frac{2b_n}{(2n)^{\log_3(\sqrt{2}+1)}}\geq 1
\]
(where $\log_3$ denotes the base-$3$ logarithm) and conjectured that equality holds; this conjecture was recently proved by Coons~\cite{C2017}, using the method used by Coons and Tyler~\cite{CT2014}, see also Coons and the author~\cite{CS2017} .

Considering the first values of Northshield's sequence, we note the following apparent symmetry: $b_{3^k+m}=b_{3^{k+1}-m}$ for $m\leq 3^k$, proved by Northshield~\cite{N2015}.
Moreover, a less apparent property is present, which is the subject of this paper:
for example, we have $b_{11}=b_{19}=5$ and $b_{29}=b_{55}=7$. The indices on both sides of these identities are related via \emph{digit reversal} in base $3$.
More precisely, for the proper base-$q$ expansion of an integer $n\geq 1$,
$n=\varepsilon_\nu q^\nu+\cdots+\varepsilon_0$, we define
$n^R=\varepsilon_0q^\nu+\cdots+\varepsilon_{\nu-1}q^1+\varepsilon_\nu$.
(We will always indicate the base in which the digit reversal is performed.)
We are going to prove the following theorem.
\begin{theorem}\label{thmNreversal}
We have
\[b_n=b_{n^R},\]
where the digit reversal is performed in base $3$.
\end{theorem}
In fact, we prove the more general Theorem~\ref{thmmain} below.
For the Stern sequence $s$, such a digit reversal property (in base $2$) was pointed out by Dijkstra~\cite{D1980},\cite[pp.230--232]{D1982}.
It can be seen as the statement that the continued fractions $[k_0;k_1,\ldots,k_r]$ and $[k_r;k_{r-1},\ldots,k_0]$ have the same numerator. The close connection between Stern's sequence and continued fractions is well known (see Graham, Knuth, Patashnik~\cite[Exercise~6.50]{GKP1989}, Lehmer~\cite{L1929}, Lind~\cite{L1969}, Stern~\cite{S1858}):
if $n=(1^{k_0}0^{k_1}\cdots 1^{k_{r-2}}0^{k_{r-1}}1^{k_r})_2$,
then $s_n$ is the numerator of the continued fraction
$[k_0;k_1,\ldots,k_r]$.

Morgenbesser and the author~\cite{MS2012} proved an analogous digit reversal property for the correlation
\[\gamma_t(\vartheta)=\lim_{N\rightarrow\infty}\frac 1N
\sum_{0\leq n<N}\e\bigl(\vartheta \sigma_q(n+t)-\vartheta \sigma_q(n)\bigr),
\]
where $\e(x)=\exp(2\pi i x)$ and $\sigma_q(n)$ is the sum of digits of $n$ in base $q$ ($q\geq 2$ an integer).
That is, we proved that $\gamma_t(\vartheta)=\gamma_{t^R}(\vartheta)$,
where the digit reversal is in base $q$.
We note that the quantity $\gamma_t(\vartheta)$ satisfies (see B\'esineau~\cite{B1972})
$\gamma_0(\vartheta)=1$ and
for $0\leq k<q$ and $t\geq 0$,
\begin{equation}\label{eqnBesineauRec}
\gamma_{qt+k}(\vartheta)=\frac{q-k}{q}\e(\vartheta k)\gamma_t(\vartheta)+\frac kq \e\bigl(-\vartheta(q-k)\bigr)\gamma_{t+1}(\vartheta),\end{equation}
which we will use later.
With the help of these correlations, we proved that
$c_t=c_{t^R}$, where
\[c_t=\lim_{N\rightarrow\infty}\frac 1N\left\lvert\left\{n<N:\sigma_2(n+t)\geq \sigma_2(n)\right\}\right\rvert\]
and the digit reversal is in base $2$.
These quantities arise in a seemingly simple conjecture due to Cusick (private communication, 2015) stating that $c_t>1/2$; see the paper~\cite{DKS2016} by Drmota, Kauers, and the author for partial results on this conjecture.

Moreover, the author~\cite{S2017} recently proved a digit reversal property for \emph{Stern polynomials}.
In that paper, we define polynomials $s_n(x,y)$ by $s_1(x,y)=1$ and
\begin{equation}\label{eqnSternpolyRec}
\begin{aligned}
s_{2n}(x,y)&=s_n(x,y)\\
s_{2n+1}(x,y)&=xs_n(x,y)+ys_{n+1}(x,y)
\end{aligned}
\end{equation}
for $n\geq 1$ and prove that $s_{n^R}(x,y)=s_n(x,y)$, where the digit reversal is in base $2$.
We note that this is a generalization of the case $q=2$
 from the above-cited paper by Morgenbesser and the author, and also implies the digit reversal symmetry for Stern's diatomic sequence.

The recurrence relations for the sequences discussed above imply that we are dealing with $q$-\emph{regular sequences} in the sense of Allouche and Shallit~\cite{AS1992}, where $q\in\{2,3\}$ and the underlying ring is $\mathbb Z,\mathbb Z[\sqrt{2}],\mathbb C$, and $\mathbb Z[x,y]$ respectively.
\section{Results}
We will derive the digit reversal property stated in the introduction as a corollary of the following theorem.
\begin{theorem}\label{thmmain}
Let $q\geq 2$ be an integer.
Assume that $(x_n)_{n\geq 0}$ is a sequence of complex numbers having the following properties:
\begin{enumerate}
\item
There exist complex $2\times 2$-matrices $A(0),\ldots,A(q-1)$ and complex numbers $\alpha,\beta$ such that for all $n\geq 1$
we have the representation
\[x_n=(1\,,0)A(\varepsilon_0)\cdots A(\varepsilon_{\nu-1})\left(\begin{matrix}\alpha\\\beta\end{matrix}\right),\]
where $(\varepsilon_{\nu-1}\cdots \varepsilon_0)_q$ is the proper $q$-ary expansion of $n$.
\item
Write 
\[A(\varepsilon)=\left(\begin{matrix}a_1(\varepsilon)&a_2(\varepsilon)\\a_3(\varepsilon)&a_4(\varepsilon)\end{matrix}\right).\]
There exist complex numbers $a\neq 0,b\neq 0,c,d$ such that $ad-bc=1$ and
\begin{equation}\label{eqnessential}
ab\bigl(a_1(\varepsilon)\beta-a_3(\varepsilon)\alpha-a_4(\varepsilon)\beta\bigr)+a_2(\varepsilon)(\beta+2bc\beta-cd\alpha)=0
\end{equation}
for all $\varepsilon\in\{0,\ldots,q-1\}$.
\end{enumerate}
Then $x_n=x_{n^R}$, where the digit reversal is in base $q$.
\end{theorem}
From this, we derive a digit reversal property for a family of ($3$-regular) sequences, among which we find Northshield's sequence.
\begin{corollary}\label{corParam}
Let $\tau,\sigma$ be complex numbers and set $\omega=1-\sigma^2+\tau\sigma$.
Assume that the sequence $(a_n)_{n\geq 0}$ satisfies
$a_0=0$, $a_1=1$,
and for $n\geq 0$,
$a_{3n}=a_n$ and
\begin{align*}
a_{3n+1}&=\tau\cdot a_n+a_{n+1},\\
a_{3n+2}&=\omega\cdot a_n+\sigma\cdot a_{n+1}.
\end{align*}
Then $a_n=a_{n^R},$
where the digit reversal is in base $3$.
\end{corollary}
In particular, $\tau=\sigma=\sqrt{2}$ yields Theorem~\ref{thmNreversal}.

\begin{proof}[Proof of Corollary~\ref{corParam}]
We note that we can express the sequence $a$ in the form given in the theorem:
set $A(0)=\left(\begin{smallmatrix}1&0\\\tau&1\end{smallmatrix}\right),
A(1)=\left(\begin{smallmatrix}\tau&1\\\omega&\sigma\end{smallmatrix}\right),
A(2)=\left(\begin{smallmatrix}\omega&\sigma\\0&1\end{smallmatrix}\right)$.
Then
\[
(a_{3n+\varepsilon}\,,a_{3n+\varepsilon+1})
=
A(\varepsilon)\left(\begin{matrix}a_n\\a_{n+1}\end{matrix}\right)
\]
for all $n\geq 0$ and $\varepsilon\in\{0,1,2\}$.
Choosing $\alpha=a_0=0$ and $\beta=a_1=1$, we see that the first condition in Theorem~\ref{thmmain} is satisfied.
To verify equation~\eqref{eqnessential}, we set
\[\left(\begin{matrix}a&b\\c&d\end{matrix}\right)=\left(\begin{matrix}1&1\\\frac{\sigma-\tau-1}2&\frac{\sigma-\tau+1}2\end{matrix}\right)\]
and evaluate the left hand side for each of the three matrices $A(\varepsilon)$. This yields $0$ in each case after a short calculation.
\end{proof}

Moreover, we want to re-prove the known digit reversal properties presented in the introduction. Whereas the proof of Corollary~\ref{corMS} is essentially the same as the one given by Morgenbesser and the author~\cite{MS2012},
the proof of Corollary~\ref{corstern} is different from the one given by the author~\cite{S2017}.
\begin{corollary}\label{corMS}
Let $\vartheta\in\mathbb R$ and assume that
\[\gamma_t(\vartheta)=\lim_{N\rightarrow\infty}\frac 1N
\sum_{0\leq n<N}\e\bigl(\vartheta \sigma_q(n+t)-\vartheta \sigma_q(n)\bigr)
\]
for all integers $t\geq 0$.
Then $\gamma_t(\vartheta)=\gamma_{t^R}(\alpha)$.
\end{corollary}

\begin{proof}
Set
\[A(\varepsilon)=\left(\begin{matrix}
\frac{q-\varepsilon}{q}\e(\vartheta \varepsilon)&\frac{\varepsilon}q \e(-\vartheta(q-\varepsilon))\\
\frac{q-\varepsilon-1}q\e(\vartheta(\varepsilon+1))&\frac{\varepsilon+1}{q}\e(-\vartheta(q-\varepsilon-1))
\end{matrix}\right)\]
and $\alpha=\gamma_0(\vartheta)=1$, $\beta=\gamma_1(\vartheta)=(q-1)/(q\e(-\vartheta)-\e(-\vartheta q))$.
Then by~\eqref{eqnBesineauRec} the first condition in Theorem~\ref{thmmain} is satisfied.
As in the paper~\cite{MS2012} by Morgenbesser and the author, set
\[\left(\begin{matrix}a&b\\c&d\end{matrix}\right)=
\left(\begin{matrix}1&\overline \beta\\0&1\end{matrix}\right).\]
Inserting these values into~\eqref{eqnessential} and multiplying the equation by $\lvert u\rvert^2 q(q-1)$, the proof is complete after a straightforward calculation.
\end{proof}

\begin{corollary}\label{corstern}
Assume that $x$ and $y$ are complex numbers and that the sequence $(z_n)_{n\geq 0}$ satisfies $z_{2n}=z_n$ and $z_{2n+1}=xz_n+yz_{n+1}$ for all $n\geq 1$.
Then $z_n=z_{n^R}$, where the digit reversal is in base $2$.
In particular, Stern's diatomic sequence and the Stern polynomials defined by equation~\eqref{eqnSternpolyRec} satisfy a digit reversal symmetry.
\end{corollary}
\begin{proof}
Without loss of generality we may assume that $z_1=1$. The general case follows from rescaling the sequence $z$ if $z_1\neq 0$; if $z_1=0$, then $z_n=0$ for all $n\geq 1$ and the statement is trivial.
Moreover, we may assume that $x,y\not\in\{0,1\}$. For the other cases, note that $z_n$, for a given $n$, depends in a continuous way on both $x$ and $y$, so that the statement follows by approximation.

Set $A(0)=\left(\begin{smallmatrix}1&0\\x&y\end{smallmatrix}\right)$ and
$A(1)=\left(\begin{smallmatrix}x&y\\0&1\end{smallmatrix}\right)$
and $\alpha=(1-y)/x$, $\beta=1$.
Then $z$ satisfies the first condition in Theorem~\ref{thmmain}.
Let $\gamma$ be a root of $4\gamma^2=\frac yx\frac{1-y}{1-x}$ and set
\[\left(\begin{matrix}a&b\\c&d\end{matrix}\right)=\left(\begin{matrix}\gamma&1\\-\frac 12&\frac 1{2\gamma}\end{matrix}\right).\]
The verification of~\eqref{eqnessential} is straightforward.
\end{proof}

\section{Proof of Theorem~\ref{thmmain}}
We closely follow the proof of Theorem~1 in the paper~\cite{MS2012} by Morgenbesser and the author.
Consider the following property:
Assume that $\nu\geq 0$ and $\varepsilon_i\in\{0,\ldots,q-1\}$ for $0\leq i<\nu$.
Then
\[(a\,,0)S^{-1}A(\varepsilon_0)\cdots A(\varepsilon_{\nu-1})\left(\begin{matrix}\alpha\\\beta\end{matrix}\right)
=
(1\,,0)A(\varepsilon_{\nu-1})\cdots A(\varepsilon_0)S\left(\begin{matrix}d\alpha-b\beta\\0\end{matrix}\right)
\]
and
\[(0\,,b)S^{-1}A(\varepsilon_0)\cdots A(\varepsilon_{\nu-1})\left(\begin{matrix}\alpha\\\beta\end{matrix}\right)
=
(1\,,0)A(\varepsilon_{\nu-1})\cdots A(\varepsilon_0)S\left(\begin{matrix}0\\-c\alpha+a\beta\end{matrix}\right),
\]
where the empty product is to be interpreted as the identity matrix.
Summing these equations and noting that $(a\,,b)S^{-1}=(1\,,0)$ and $S(d\alpha-b\beta\,,-c\alpha+a\beta)^T=(\alpha\,,\beta)^T$, which follows from $ad-bc=1$, we obtain the statement of Theorem~\ref{thmmain}.

We will show the above identities by induction on $\nu$. For $\nu=0$ we have
$(a\,,0)S^{-1}(\alpha\,,\beta)^T
=
(1\,,0)S(d\alpha-b\beta\,,0)$
and
$(0\,,b)S^{-1}(\alpha,\beta)^T
=
(1\,,0)S(0\,,-c\alpha+a\beta)^T$,
which is trivial to check.
Assume now that $\nu\geq 1$. We set
\[
\left(\begin{matrix}\mathfrak{a} \\\mathfrak{b}\end{matrix}\right) =
S^{-1} A(\varepsilon_1)\cdots A(\varepsilon_{\nu-1})
\left(\begin{matrix}\alpha\\\beta\end{matrix}\right)
\ \text{and}\ %
(\mathfrak a'\,,\mathfrak b') = (1\,,0) A(\varepsilon_{\nu-1}) \cdots A(\varepsilon_1) S.
\]
By the induction hypothesis we obtain 
\begin{equation}\label{eqnfrakeq}
a\mathfrak a=(d\alpha-b\beta)\mathfrak a'\quad\textrm{and}\quad b\mathfrak b=(-c\alpha+a\beta)\mathfrak b'.
\end{equation}
We need to show that
\begin{equation}\label{eqnInduction1}
(a,\, 0) S^{-1} A(\varepsilon_0) S
\left(\begin{matrix}\mathfrak{a} \\\mathfrak{b}\end{matrix}\right)
= (\mathfrak a'\,,\mathfrak b') S^{-1} A(\varepsilon_0) S
\left(\begin{matrix}d\alpha-b\beta\\0\end{matrix}\right)
\end{equation}
and
\begin{equation}\label{eqnInduction2}
(0,\, b) S^{-1} A(\varepsilon_0) S
\left(\begin{matrix}\mathfrak{a} \\\mathfrak{b}\end{matrix}\right)
= (\mathfrak a'\,,\mathfrak b') S^{-1} A(\varepsilon_0) S
\left(\begin{matrix}0\\-c\alpha+a\beta\end{matrix}\right).
\end{equation}

For brevity, we set
\[\left(\begin{matrix}s_1(\varepsilon)&s_2(\varepsilon)\\s_3(\varepsilon)&s_4(\varepsilon)\end{matrix}\right)=S^{-1}A(\varepsilon)S.\] 
It is easily seen, using~\eqref{eqnfrakeq}
and the restrictions $a\neq 0$, $b\neq 0$,
that each of equations~\eqref{eqnInduction1} and~\eqref{eqnInduction2} follows from
\begin{equation}\label{eqnSrelation}
a(-c\alpha+a\beta)s_2(\varepsilon_0)=b(d\alpha-b\beta)s_3(\varepsilon_0).
\end{equation}
Noting that $s_2(\varepsilon)=bda_1(\varepsilon)+d^2a_2(\varepsilon)-b^2a_3(\varepsilon)-bda_4(\varepsilon)$ and $s_3(\varepsilon)=-aca_1(\varepsilon)-c^2a_2(\varepsilon)+a^2a_3(\varepsilon)+aca_4(\varepsilon)$ and using $ad-bc=1$,~\eqref{eqnSrelation} reduces to~\eqref{eqnessential} after some short calculation.
This finishes the proof of Theorem~\ref{thmmain}.


\begin{thebibliography}{10}

\bibitem{AS1992}
J.-P. Allouche and J. Shallit,
\newblock The ring of {$k$}-regular sequences,
\newblock {\it Theoret. Comput. Sci.}, {\bf 98} (1992), 163--197.

\bibitem{B1972}
J. B\'esineau,
\newblock Ind\'ependance statistique d'ensembles li\'es \`a la fonction ``somme des chiffres'',
\newblock {\it Acta Arith.}, {\bf 20} (1972), 401--416.

\bibitem{C2017}
M. Coons,
\newblock Proof of {N}orthshield's conjecture concerning an analogue of
  {S}tern's sequence for $\mathbb {Z}[\sqrt{2}]$,
\newblock preprint, 2017,
\newblock http://arxiv.org/abs/1709.01987.

\bibitem{CS2017}
M. Coons and L. Spiegelhofer,
\newblock The maximal order of hyper-({$b$}-ary)-expansions.
\newblock {\it Electron. J. Combin.}, {\bf 24} (2017), Paper 1.15.

\bibitem{CT2014}
M. Coons and J. Tyler,
\newblock The maximal order of {S}tern's diatomic sequence.
\newblock {\it Mosc. J. Comb. Number Theory}, {\bf 4} (2014), 3--14.

\bibitem{D1982}
E.~W. Dijkstra.
\newblock {\it Selected writings on computing: a personal perspective}.
\newblock Texts and Monographs in Computer Science. Springer-Verlag, New York,
  1982.
\newblock Including a paper co-authored by C. S. Scholten.

\bibitem{D1980}
E.~W. Dijkstra.
\newblock Problem 563.
\newblock {\it Nieuw Archief voor Wiskunde}, {\bf XXVII} (1980), p.115.

\bibitem{DKS2016}
M. Drmota, M. Kauers, and L. Spiegelhofer.
\newblock On a conjecture of {C}usick concerning the sum of digits of
  {$n$} and {$n+t$}.
\newblock {\it SIAM J. Discrete Math.}, {\bf 30} (2016), 621--649.

\bibitem{GKP1989}
R.~L. Graham, D.~E. Knuth, and O. Patashnik.
\newblock {\it Concrete {M}athematics: a {F}oundation for {C}omputer
  {S}cience}.
\newblock Addison--Wesley, 1989.

\bibitem{L1929}
D.~H. Lehmer.
\newblock On {S}tern's diatomic series.
\newblock {\it Amer. Math. Monthly}, {\bf 36} (1929), 59--67.

\bibitem{L1969}
D.~A. Lind.
\newblock An extension of {S}tern's diatomic series.
\newblock {\it Duke Math. J.}, {\bf 36} (1969), 55--60.

\bibitem{MS2012}
J.~F. Morgenbesser and L. Spiegelhofer.
\newblock A reverse order property of correlation measures of the sum-of-digits
  function.
\newblock {\it Integers}, {\bf 12} (2012), Paper No. A47.

\bibitem{N2015}
S.~Northshield.
\newblock An analogue of {S}tern's sequence for {$\Bbb Z[\sqrt 2]$}.
\newblock {\it J. Integer Seq.}, {\bf 18} (2015), Article 15.11.6.

\bibitem{S2017}
L. Spiegelhofer.
\newblock A digit reversal property for {S}tern polynomials.
\newblock preprint, 2017,
\newblock http://arxiv.org/abs/1610.00108.

\bibitem{S1858}
M.~A. Stern.
\newblock Ueber eine zahlentheoretische {F}unktion.
\newblock {\it J. {R}eine {A}ngew. {M}ath.}, {\bf 55} (1858), 193--220.

\end{thebibliography}
\end{document}